\documentclass[a4paper]{amsart}
\usepackage{amsmath,amssymb,amsfonts,amsthm}
\usepackage{mathtools}
\usepackage{verbatim}
\usepackage{graphicx}
\usepackage{color}
\usepackage{enumerate}
\usepackage{tikz}
\usetikzlibrary{calc,decorations.markings}
\usetikzlibrary{arrows.meta}
\usepackage{pgfplots}
\pgfplotsset{compat=1.17} 
\usepackage[T1]{fontenc}
\usepackage{hyperref}
\usepackage{microtype}
  
\usepackage{fancyhdr}
\fancypagestyle{tp}{
\fancyhead{}
\fancyhead[C]{This is the final author manuscript. The version of
  record, subject to  editorial changes, will appear open access in Ars
  Mathematica Contemporanea. \\
  \href{https://doi.org/10.26493/1855-3974.3793.1fa}{\texttt{\textsc{doi}:10.26493/1855-3974.3793.1fa}}
}}

\hypersetup{
    pdftitle={Biggs tree groups},    
    pdfauthor={Christopher Cashen},     
    pdfkeywords={Biggs groups, Cayley graphs, large girth, primitive permutation groups}, 
    colorlinks=true,       
    linkcolor=black,          
    citecolor=black,        
    filecolor=black,      
    urlcolor=black           
  }
\theoremstyle{plain}
\newtheorem{theorem}{Theorem}[section]
\newtheorem{lemma}{Lemma}[section]
\newtheorem{proposition}{Proposition}[section]
\newtheorem{corollary}{Corollary}[section]
\newtheorem*{summarytheorem}{Main Theorem}
\newtheorem{question}{Question}

\theoremstyle{remark}

\theoremstyle{definition}
\newtheorem{definition}{Definition}[section]

\def\makeautorefname#1#2{\expandafter\def\csname#1autorefname\endcsname{#2}}
\let\fullref\autoref

\makeautorefname{theorem}{Theorem} 
\makeautorefname{lemma}{Lemma} 
\makeautorefname{proposition}{Proposition} 
\makeautorefname{corollary}{Corollary} 
\makeautorefname{definition}{Definition}
\makeautorefname{example}{Example}
\makeautorefname{section}{Section}
\makeautorefname{subsection}{Section}
\makeautorefname{subsubsection}{Section}
\makeautorefname{claim}{Claim}

\makeatletter 
\let\c@lemma=\c@theorem 
\makeatother
\makeatletter 
\let\c@proposition=\c@theorem 
\makeatother
\makeatletter 
\let\c@corollary=\c@theorem 
\makeatother
\makeatletter 
\let\c@definition=\c@theorem 
\makeatother
\makeatletter 
\let\c@example=\c@theorem 
\makeatother
\makeatletter
\@addtoreset{claim}{proposition}
\makeatother
\makeatletter
\@addtoreset{claim}{lemma}
\makeatother
\mathtoolsset{centercolon} 

\newcommand{\bdry}{\partial} 
\newcommand{\act}{\curvearrowright} 
\newcommand{\from}{\colon\thinspace} 
\newcommand{\bp}{o}
\renewcommand{\setminus}{-}
\DeclareMathOperator{\stab}{Stab}
\DeclareMathOperator{\cay}{Cay}
\DeclareMathOperator{\girth}{girth}
\DeclareMathOperator{\diam}{diam}
\DeclareMathOperator{\sym}{Sym}
\DeclareMathOperator{\alt}{Alt}
\DeclareMathOperator{\lcm}{lcm}
\DeclareMathOperator{\lcmleq}{\underline{lcm}}

\DeclareMathOperator{\oddlcmleq}{\underline{oddlcm}}
\DeclareMathOperator{\pgammal}{P\Gamma L}
\DeclareMathOperator{\pgl}{PGL}
\DeclareMathOperator{\psl}{PSL}
\DeclareMathOperator{\aut}{Aut}
\DeclareMathOperator{\meo}{meo}
\DeclareMathOperator{\ord}{ord}
\DeclareMathOperator{\Landau}{\Lambda}
\renewcommand{\ln}{\operatorname{log}}
\newcommand{\cleq}{\preceq}
\newcommand{\cgeq}{\succeq}
\newcommand{\ceq}{\asymp}

\title{Biggs tree groups}
\date{July 20, 2026}

\author{Christopher H.\ Cashen} 
\address{TU Wien\\  Institute of Discrete Mathematics and Geometry\\Wiedener Hauptstrasse~8-10\\1040 Vienna \\Austria\\\href{https://orcid.org/0000-0002-6340-469X}{\includegraphics[scale=.75]{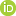}~0000-0002-6340-469X}}
\email{christopher.cashen@tuwien.ac.at}
\begin{document}
\begin{abstract}
 Biggs gave an explicit construction, using finite colored trees, of finite permutation groups whose Cayley graphs have valence $C$ and girth tending to infinity as the radius $R$ of the tree tends to infinity.
We show that when the number of colors is at least 3, the group so presented contains the full
alternating group on the vertices of the tree.

This gives, for each $C\geq 3$, an infinite family of pairs $(G_{C,R},S_{C,R})$
such that $G_{C,R}$ is an alternating or symmetric group, $S_{C,R}$ is a
generating set of $G_{C,R}$ of size $C$ with an explicit permutation
description of its generators, and such that the sequence of Cayley
graphs $\cay(G_{C,R},S_{C,R})$ has constant valence $C$
and girth tending to infinity as $R$
tends to infinity.

\end{abstract}
\maketitle
\thispagestyle{tp}

\section{Introduction}
Let $C$ be a finite set of colors, and let $\Gamma$ be a finite
simplicial graph such that each edge is assigned a color from $C$, and
such that no two edges incident to a common vertex have the same
color (a \emph{proper coloring}).
We may assume every color of $C$ is actually used on some edge. 
Then each $c\in C$ determines a nontrivial permutation of the vertices of $\Gamma$ by
swapping the two vertices of each edge of color $c$.
Let $G$ be the permutation group generated by $C$.

Biggs \cite{MR342422} calls the colored graph a `picture' and the
group $G$ the `story of the picture'.
See also the survey \cite{Big98}.
Later works, eg \cite{MR1127339},  refer to $G$ as a `Biggs group',
although this is a bit of a misnomer, as it is the
choice of generating set that is novel, not the group itself. 

By construction, each $c\in C$ is an involution, since it is a product
of disjoint transpositions, so the Cayley graph $\mathrm{Cay}(G,C)$ of $G$ with respect to
the generating set $C$ is a $C$--valent\footnote{Whenever $C$ appears
  where an integer is expected, interpret the notation to mean $|C|$.}, undirected graph. 
Of particular interest is the case that $\Gamma=T_{C,R}$ is the closed ball
of radius $R$ centered at a root vertex $\bp$ of the $C$--regular
$C$--colored tree $T_C$.
The number of vertices of $T_{C,R}$ is
$N_{C,R}:=\frac{C(C-1)^R-2}{C-2}$ when $C>2$.
By numbering the vertices 1, \dots, $N_{C,R}$, in whatever way, we may
consider the permutation group $G_{C,R}$ as a subgroup of $\mathrm{Sym}(N_{C,R})$.
For $C=3$, Biggs \cite{Big89} showed that $\mathrm{girth}(\mathrm{Cay}(G_{C,R},C))\geq
2R+1$.
This was the first example of a family of cubic (3--valent) Cayley graphs with
unbounded girth.
The girth argument also generalizes to $C\geq 3$
\cite[Theorem~19]{dynamic_cage_survey}.

It is not hard to see (\cite[Lemma~7.1]{Big98} or
\fullref{dichrome_is_dihedral}) that $G_{2,R}$ is a finite dihedral
group.
Exoo and Jajcay \cite{ExoJaj12} note that in small examples where
$G_{C,R}$ can be explicitly computed, it turns out that when $C>2$ the
group $G_{C,R}$ is the
full symmetric group or the alternating group.
However, $|\mathrm{Sym}(N_{C,R})|$ grows
quickly, and only for the very smallest examples can the isomorphism
type of $G_{C,R}$ be checked by computer. 
The purpose of this paper is to explain the experimental observation of Exoo and
Jajcay.
Our main tasks are to use the colored tree $T_{C,R}$ to show that
$G_{C,R}$ is 2--transitive, hence primitive, and to produce some
cycles. 
Then we can apply a theorem (\fullref{jones}) of Jones that uses the Classification
of Finite Simple Groups to refine the Jordan Symmetric Group Theorem. 
This leads quickly to the conclusion
that either $G_{C,R}$ contains the alternating
group or it is a subgroup of some projective semi-linear group
of the same permutation degree.
We rule out the latter using a variety of methods,
applying to different ranges of $C$ and $R$:
\begin{enumerate}
\item Demonstrate the existence of cycles with many fixed points.
  \item Show there is no projective semi-linear group of the same
    degree.
    \item Show that $G_{C,R}$ contains an element whose order is too
      high relative to the permutation degree for it to
      occur in a projective semi-linear group. 
\end{enumerate}
The conclusion is the main theorem of this paper, whose proof occupies \fullref{sec:main_theorem}.
\begin{summarytheorem}
For $C\geq 3$:
  \[G_{C,R}\cong
    \begin{cases}
      \alt(N_{C,R})\quad \text{ if $C$
        and $R$ are even}\\
      \sym(N_{C,R})\quad \text{otherwise}
    \end{cases}\]
\end{summarytheorem}

\fullref{fig:nauru} shows $\cay(G_{C,R},C)$ for $C=\{r,g,b\}$ and
$R=1$ embedded in a torus, obtained by identifying opposite
sides of the dashed hexagon.
It is a Cayley graph of $\sym(4)$ that is isomorphic to the
Nauru graph.
The order of $G_{3,2}$ is already $10!=3628800$, so we
will not try to draw its Cayley graph. 

\begin{figure}[h]
  \centering
  \begin{tikzpicture}\tiny
    \pgfmathsetmacro{\r}{1.25};
    \pgfmathsetmacro{\x}{\r/2};
    \pgfmathsetmacro{\y}{-\r*sqrt(3)/2};
    \coordinate (O) at (0,0);
    \node (o) at (O) {$()$};
    \coordinate (R) at (0:\r);
    \node (r) at (R) {$(01)$};
    \coordinate (G) at (240:\r);
    \node (g) at (G) {$(02)$};
    \coordinate (B) at (120:\r);
    \node (b) at (B) {$(03)$};
    \coordinate (RB) at ($(R)+(60:\r)$);
    \node (rb) at (RB) {$(013)$};
    \coordinate (RG) at ($(R)+(-60:\r)$);
    \node (rg) at (RG) {$(012)$};
    \coordinate (GR) at ($(G)+(-60:\r)$);
    \node (gr) at (GR) {$(021)$};
    \coordinate (GB) at ($(G)+(180:\r)$);
    \node (gb) at (GB) {$(023)$};
    \coordinate (BG) at ($(B)+(180:\r)$);
    \node (bg) at (BG) {$(032)$};
    \coordinate (BR) at ($(B)+(60:\r)$);
    \node (br) at (BR) {$(031)$};
    \coordinate (RBR) at ($(RB)+(120:\r)$);
    \node (rbr) at (RBR) {$(13)$};
    \coordinate (RGR) at ($(RG)+(240:\r)$);
    \node (rgr) at (RGR) {$(12)$};
    \coordinate (GBG) at ($(GB)+(120:\r)$);
    \node (gbg) at (GBG) {$(23)$};
    \coordinate (GBR) at ($(GB)+(240:\r)$);
    \node (gbr) at (GBR) {$(0231)$};
     \coordinate (GRB) at ($(GR)+(240:\r)$);
     \node (grb) at (GRB) {$(0213)$};
     \coordinate (GRBR) at ($(GRB)+(180:\r)$);
     \node (grbr) at (GRBR) {$(02)(13)$};
     \coordinate (GRBG) at ($(GRB)+(-60:\r)$);
     \node (grbg) at (GRBG) {$(132)$};
      \coordinate (RGRB) at ($(RGR)+(-60:\r)$);
      \node (rgrb) at (RGRB) {$(03)(12)$};
      \coordinate (RGRBG) at ($(RGRB)+(240:\r)$);
      \node (rgrbg) at (RGRBG) {$(0321)$};
       \coordinate (RGB) at ($(RG)+(0:\r)$);
       \node (rgb) at (RGB) {$(0123)$};
        \coordinate (RGBR) at ($(RGB)+(-60:\r)$);
        \node (rgbr) at (RGBR) {$(123)$};
        \coordinate (RGRBR) at ($(RGRB)+(0:\r)$);
        \node (rgrbr) at (RGRBR) {$(0312)$};
    \coordinate (RBG) at ($(RB)+(0:\r)$);
    \node (rbg) at (RBG) {$(0132)$};
     \coordinate (RGBG) at ($(RGB)+(60:\r)$);
    \node (rgbg) at (RGBG) {$(01)(23)$};
   \draw[ultra thick, red] (o)--(r) (g)--(gr) (b)--(br) (rb)--(rbr)
   (rg)--(rgr) (gb)--(gbr) (grb)--(grbr) (rgb)--(rgbr) (rgrb)--(rgrbr)
   (bg)--($(BG)+(120:\r/2)$) (rgrbg)--($(RGRBG)+(-60:\r/2)$)
   (rbg)--($(RBG)+(60:\r/2)$) (grbg)--($(GRBG)+(240:\r/2)$)
   (gbg)--($(GBG)+(180:\r/2)$) (rgbg)--($(RGBG)+(0:\r/2)$);
   \draw[ultra thick, green] (o)--(g) (r)--(rg) (b)--(bg) (gr)--(rgr)
   (gb)--(gbg) (grb)--(grbg) (rgrb)--(rgrbg) (rb)--(rbg) (rgb)--(rgbg)
   (br)--($(BR)+(120:\r/2)$) (rgrbr)--($(RGRBR)+(-60:\r/2)$)
   (rbr)--($(RBR)+(60:\r/2)$) (grbr)--($(GRBR)+(240:\r/2)$)
   (gbr)--($(GBR)+(180:\r/2)$) (rgbr)--($(RGBR)+(0:\r/2)$);
   \draw[ultra thick, blue] (o)--(b) (r)--(rb) (g)--(gb) (br)--(rbr)
   (bg)--(gbg) (gr)--(grb) (gbr)--(grbr) (rgr)--(rgrb) (grbg)--(rgrbg)
   (rg)--(rgb) (rgbr)--(rgrbr) (rbg)--(rgbg);
   \begin{scope}[shift={(\x,\y)}]
     \draw[dashed] (30:3.5*\r)--(90:3.5*\r)--(150:3.5*\r)--(210:3.5*\r)--(270:3.5*\r)--(330:3.5*\r)--(30:3.5*\r);
   \end{scope}
  \end{tikzpicture}
  \caption{$\cay(G_{\{r,g,b\},1},\{r,g,b\})$, vertices labeled by
    permutations 
    of vertices of $T_{3,1}$ with $\bp=0$, $\bp.r=1$,
    $\bp.g=2$, and $\bp.b=3$.}
  \label{fig:nauru}
\end{figure}

\subsection{Comparison to other results}
There is a long history in graph theory and in group theory of producing families of bounded valence graphs realizing extreme properties. We focus on girth and diameter in Cayley graphs of finite groups, with fixed number of generators, hence fixed valence.
It is known that all possible combinations of valence and girth do indeed occur \cite{MR2846070,MR3070127}.
There is a tension between girth and diameter in the sense that the na\"ive estimates (Moore bounds) comparing balls in the graphs to balls in a tree of the same valence show that diameter grows at least logarithmically and girth grows at most logarithmically in terms of number of vertices of the graph. 
This makes it a challenge to produce families in which the diameter-to-girth ratio stays bounded.

The property of being an \emph{expander family} implies logarithmic diameter growth, but does not guarantee even that girth is unbounded.
In the Cayley graph setting, most constructions of expander families use finite groups of Lie type.
The faster growth of alternating groups makes finding expander families of alternating groups harder, although it is possible, by work of Kassabov \cite{MR2342639}.
Unlike many of the Lie type expander families, ie those of Margulis \cite{MR671147}, Lubotzky, Phillips, and Sarnak \cite{MR963118}, and Arzhantseva and Biswas \cite{ArzBis22}, Kassabov's construction does not simultaneously produce large girth.
See \cite[Introduction]{ArzBis22} for a discussion and extensive bibliography.

There do exist choices of generating sets of alternating and symmetric groups that produce large girth families, due to work of Gamburd, Hoory, Shahshahani, Shalev, and Vir\'ag \cite{MR2532876}, but these are existence results using probabilistic methods.

We compare the behavior of the groups $G_{C,R}$ with these contexts.

\subsection*{Random Cayley graphs}
Gamburd et al \cite{MR2532876} show that if one chooses a generating set of fixed size uniformly at random in a sufficiently large finite group from within certain families, then almost surely
the resulting Cayley graphs have high girth.
For the symmetric group their lower bound is roughly the square root of the logarithm of the order of the group \cite[Theorem~3]{MR2532876}, although they conjecture the square root can be removed. 
These results are probabilistic, not constructive.

In contrast, the Biggs tree construction is entirely explicit: the generators are the colors in $C$, and their action is described combinatorially via tracks in the colored tree, see \fullref{sec:tracks}.
While we do not know the precise girth, we give an upper bound in \fullref{girth_upper_bound} that we conjecture is the true girth.
This upper bound is exponential in $R$.
Coupled with the main theorem saying that $G_{C,R}$ is the full
symmetric or alternating group, this would say that the girth $g$ satisfies $g\asymp\frac{\ln(|G_{C,R}|)}{\ln(\ln(|G_{C,R}|))}$.
This falls in between the proven and conjectured asymptotic girth growth rates 
from Gamburd, et al, and has the advantage of coming from an explicit construction. 

\subsection*{Symmetric/alternating expanders}
Kassabov~\cite{MR2342639} constructed bounded valence expander families in alternating and symmetric groups. 
He writes that the generating sets are `huge', but with more work they
can be replaced by generating sets of size at most 20 \cite[Remark~5.1]{MR2342639}.
The structure of the paper is to prove expansion properties with respect to generating sets that are orders of magnitude larger and then do further work to control the ratio of Kazhdan constants when passing to a suitably well behaved smaller generating set.
Although a stated goal of this paper is to construct `explicit' generating sets, it makes extensive use of the representation theory of the symmetric groups.
It also makes no attempt to control girth. 

In contrast, the Biggs tree groups can be generated with as few as
three generators, and these generators can be written down explicitly
as permutations.
As far as we know, these are the first explicit families of Cayley graphs of symmetric groups with unbounded girth. 
We have not considered expansion properties of these Cayley graphs in this paper. 
One might hope that the explicit nature of the construction will allow for further advances in this direction, without the heavy algebraic requirements. 

\subsection*{Diameter-to-girth ratio}
Biggs tree groups have been mentioned as a possible candidate for
constructing bounded diameter-to-girth ratio families.
We show in \fullref{not_dg_bounded} that this cannot work; the
Moore diameter lower bound is on the order of $R(C-1)^R$, while the
upper girth bound is on the order of $(C-1)^R$, so the diameter-to-girth ratio
tends to infinity with $R$.

\subsection*{Residual constructions}
Wilton \cite{Wil12} proved that if $F_r$ is the free group of rank
$r\geq 2$ then for every finitely generated, infinite index subgroup $H\leq F_r$ and
every finite set $K\subset F_r\setminus H$ there is a surjection
$\phi$ of $F_r$ onto a finite alternating group in which
$\phi(k)\notin\phi(H)$ for all $k\in K$.
Taking $H=\{1\}$ and $K$ to be the set of all words of length less
than $g$ in
$F_r$ with respect to a fixed basis, 
this implies that the generating set of the alternating group given
by the $\phi$-image of the basis yields a Cayley graph of girth at least
$g$.
The proof is by constructing certain finite graphs covering the
$r$-rose. There are choices involved.
In principle, if one makes those choices
explicit, then this construction could be used to deduce explicit
generating sets of fixed size $r$ yielding Cayley graphs of alternating groups with
girth going to infinity.
The generators will not be involutions. We have not explored the
relation between girth and permutation degree of these constructions,
or how that relation compares to the case of Biggs tree groups. 

\subsection{Notation}
For positive functions $f(x)$ and $g(x)$, the notation $f\sim g$ means
$\lim_{x\to\infty}f(x)/g(x)=1$, while $f=o(g)$ means
$\lim_{x\to\infty}f(x)/g(x)=0$.
The notation $f\cleq g$ means there
exists a positive constant $A$ such that $f(x)\leq Ag(x)$ for all sufficiently
large $x$. The notation $f\cgeq g$ means $g\cleq f$, and $f\ceq g$
means $f\cleq g$ and $f\cgeq g$.

\section{Primitive permutation groups containing a cycle}\label{sec:jones}
The \emph{degree} of a permutation group is the size of the set that
it permutes.
It is \emph{$m$--transitive} if it acts transitively on the set of distinct
$m$--tuples.
It is \emph{sharply $m$--transitive} if it is $m$--transitive and for
any given pair of $m$--tuples there is a unique element taking one
to the other.
It is \emph{primitive} if it does not preserve any
nontrivial partition of the set upon which it acts.

A 2--transitive permutation group of degree $n>2$ is primitive, since a nontrivial
partition of $\{1,\dots,n\}$ contains at least one non-singleton subset
$A\supset\{a_1,a_2\}$ and at least one other subset $B\supset\{b\}$, but 2--transitivity
gives a permutation $\sigma$ with $a_1.\sigma=a_1$ and $a_2.\sigma=b$.

\begin{theorem}[{Excerpt\footnote{The $k=1,2$ cases do not occur in
      our arguments, so we have omitted their subcases.} of \cite[Theorem~1.2]{Jon14}}]\label{jones}
  If $G$ is a primitive permutation group of degree $n$ containing a
  nontrivial cycle $g$ then either $G$ contains $\mathrm{Alt}(n)$ or
  $g$ fixes $0\leq k\leq 2$ points and one of the following is true:
  \begin{enumerate}
  \item[(0)]$k=0$ and:
  \begin{enumerate}
  \item $C_p\leq G\leq \mathrm{AGL}_1(p)$ for $n=p$ prime.\label{item:affine}
    \item $\mathrm{PGL}_d(q)\leq G\leq
      \mathrm{P}\Gamma\mathrm{L}_d(q)$ with $d\geq 2$, $n=(q^d-1)/(q-1)$, and $q$ a
      prime power.\label{item:projective}
      \item $G$ is one of $L_2(11)$, $M_{11}$, or $M_{23}$ with $n=11, 11, 23$, respectively. \label{item:sporadic}
  \end{enumerate}
\item $k=1$
  \item $k=2$ 
  \end{enumerate}
\end{theorem}

\section{Tracking the action}\label{sec:tracks}
By a \emph{word} we mean either an empty string or a string of letters from $C$ such that no $c\in C$ appears twice in
succession.
The tree $T_C$ has the property that given a word $w$ with letters in
$C$ and a vertex $v\in T_C$ there is a unique edge path starting at
$v$ such that the sequence of edge colors along the path is $w$.
We would like a similar property in $T_{C,R}$.
A convenient way to arrange this is to augment $T_{C,R}$ by adding
mirrored half-edges at the leaves.
A \emph{mirrored half-edge} is the quotient of a single edge by an
involution that exchanges its vertices and fixes its midpoint.
Geometrically, it is isometric to a segment of length $1/2$ with one
endpoint identified with one of the vertices of $T_{C,R}$ and the
other designated the `mirror'.
For each leaf $v$ of $T_{C,R}$ there is a unique incident edge.
If the color of that edge is $c_v\in C$ then for each  $c\in
C\setminus\{c_v\}$ we add a mirrored half-edge colored $c$ to $T_{C,R}$ at $v$.
Formally one could say that $T_{C,R}$ with the added colored mirrored
half-edges is the quotient of $T_C$, with unit edge lengths, by a
color preserving isometric group action
that is free on vertices but inverts some edges, and the `mirrors' are
the images of the edge midpoints with $\mathbb{Z}/2\mathbb{Z}$ stabilizers.
Alternatively, one could subdivide the edges of $T_{C}$ with nontrivial
stabilizer first so that the action is free on edges and some vertices
have $\mathbb{Z}/2\mathbb{Z}$ stabilizers and the quotient is actually a graph,
with some edges of length 1 and some of length $1/2$.

These two points of view give the same geometric space, but we prefer
the former and will reserve the word `vertex' for the original
vertices of $T_{C,R}$.
This allows us to make statements like: for every vertex $v$ of $T_{C,R}$
there is a color preserving isometry between the open ball of radius $1/2$
around $v$ and the open ball of radius $1/2$ about any vertex of
$T_C$.
There is a picture in \fullref{fig:rainbow_cycle}.
For the purposes of this paper the important point is the following convention: if $v$ is a vertex of $T_{C,R}$ with an incident mirrored
half-edge colored $c$, then the path starting at $v$ corresponding to
the word $c$ is the path that travels from $v$ to the mirror of the
$c$--colored mirrored half-edge and then travels back to $v$. 
With this convention, for every word $w$ with letters in
$C$ and a vertex $v\in T_{C,R}$ there is a unique (edge/mirrored
half-edge) path starting at
$v$ such that the sequence of colors along the path is $w$.
The path is locally injective except at the points at which it reflects off
of a mirror.

A more typical way to define colored edge paths in a colored graph
would be to make the path stop and wait at a vertex if there is no
incident edge of the correct color.
For our application the problem with this approach is that such a path
collapses some colored subwords to a vertex, so the path does not
remember the length or the full color sequence of the original word. 
Another approach would be to add colored loops at the leaves of
$T_{C,R}$ instead of mirrored half-edges, but this gives those
vertices the wrong valence in comparison to $T_C$ and breaks the
uniqueness of the path corresponding to a color sequence, since one
has a choice of which way to go around such a loop. 

Define the \emph{height} of a vertex in $T_{C,R}$ as its distance from
the root $\bp$.
Relative to a vertex $v$, call the edges leading farther from $\bp$ the
\emph{up edges} and, if $v\neq \bp$, call the unique edge leading from
$v$ closer to $\bp$ the \emph{down edge}.
This description is for convenience of the narration, and is only
relative to a fixed vertex; we do not treat $T_{C,R}$ as a directed
tree in the graph-theoretic sense.
Indeed, since the generators of $G_{C,R}$ are involutions there is no
need to record whether traversing an edge colored $c$ in a certain
direction corresponds to $c$ or $c^{-1}$, since these are equivalent
in $G_{C,R}$. 

The right-action of $G_{C,R}$ on $T_{C,R}$ can be simply
visualized as pushing along a path: the element of $G_{C,R}$
represented by a nonempty word
$w:=c_{i_1}\cdots c_{i_n}$ sends each vertex $v$ of $T_{C,R}$ to the
unique vertex at the end of the path starting at $v$ whose color
sequence is $c_{i_1}\cdots c_{i_n}$. 
We will refer to this path as the \emph{track of the action of $w$ on $v$}.
Since we are assuming words are freely reduced, the track of the
action of $w$ on $v$ is a path in $T_{C,R}$ that is locally injective
except when it reflects off a mirror.
Again, see \fullref{fig:rainbow_cycle}.

If  $\gamma$ is a geodesic edge path in $T_{C,R}$ from $v$ to $v'$ whose edges are
colored $c_{i_1}$, $c_{i_2}$, \dots, $c_{i_n}$ then for
$w:=c_{i_1}\cdots c_{i_n}$ the track of the action of $w$ on $v$ is
exactly the geodesic $\gamma$. 
This shows $G_{C,R}\act T_{C,R}$ is transitive; we get an
element taking any vertex $v$ to any other vertex $v'$ by pushing $v$
along the unique geodesic from $v$ to $v'$ in $T_{C,R}$.

For $c\in C$, define the \emph{peripheral subgroup} $P_{\neg{c}}$ of $G_{C,R}$ to be the subgroup
generated by all of the generators except $c$.
Define $\bdry T_{C,R}$ to be the vertices at height $R$, and let 
$\bdry_cT_{C,R}$  be the vertices of $\bdry T_{C,R}$ whose
incident edge is colored $c$.

\begin{proposition}
  The stabilizer in $G_{C,R}$ of a vertex $v$ is generated by the
subgroups $\delta P_{\neg{c}}\bar\delta$ as $c$ ranges over $C$ and
$\delta$ ranges over geodesics in $T_{C,R}$ from $v$ to a vertex in $\bdry_c T_{C,R}$, and where $\bar\delta$ denotes the reverse of
the path $\delta$.
The shortest nonempty words representing elements of $\stab(v)$ have
length $1+2d(v,\bdry T_{C,R})$, which is the minimal distance from $v$
to a mirror and back. 
\end{proposition}
\begin{proof}
  Suppose $w$ is a nonempty word stabilizing $v\in T_{C,R}$.
  The track of its action is a nontrivial path that begins and ends at the same
  vertex $v$ and is locally injective except when it reflects off a
  mirror.
  Since $T_{C,R}$ is a tree, the path must hit some mirrors,
  or else it cannot return to $v$. Thus, the track decomposes as a
  concatenation of embedded segments
  $\alpha+\beta_0+\cdots+\beta_{n-1}+\gamma$ such that:
  \begin{itemize}
  \item $\alpha$ starts at $v$ and ends at a mirror.
  \item $\gamma$ starts at a mirror and ends at $v$.
    \item If $\beta_i$ exists then it starts and ends at distinct mirrors. 
    \end{itemize}

    It follows immediately that $|w|\geq
    |\alpha|+|\gamma|=1+2d(v,\bdry T_{C,R})$, and it remains only to
    show that $w$ is a product of elements of the desired form. 

    Suppose first that there are no segments $\beta_i$. 
    Then $\gamma=\bar\alpha$.
    Let $\delta$ be the initial subsegment of $\alpha$ that includes
    all of $\alpha$ except the final half-edge.
    The last edge of $\delta$ has some color $c$, and the last vertex $v'$
    is in $\bdry_{c}T_{C,R}$, so $\alpha+\gamma$ follows the unique
    geodesic $\delta$ from $v$ to $v'$, reflects off one of the
    mirrors at distance $1/2$ from $v'$,
    and then travels along $\bar\delta$ back to $v$.
    This says $w$ is a word of the form $\delta p\bar\delta$ for some
    single letter $p \in C\setminus \{c\}$.
    Hence, the group element represented by $w$ is an element of
    $\delta P_{\neg c}\bar\delta$.
    Now suppose there are some $\beta$ segments. 
 Consider the geodesic segment $\beta_0$.
 It contains a unique vertex $v'$ closest to $v$.
 Let $\beta_0'$ be the geodesic in $T_{C,R}$ consisting of the initial
 subsegment of $\beta_0$ until the point $v'$, followed by the
 geodesic from $v'$ to $v$. 
 Let $\beta_0''$ be the geodesic in $T_{C,R}$ that follows the geodesic
 from $v$ to $v'$ and then the terminal subsegment of $\beta_0$
 starting from $v'$.
 Then $\alpha+\beta_0'$ is as in the previous paragraph, so is
 represented by a conjugate of a parabolic.
 Iteratively repeat the argument for $\beta_0''+\beta_1+\cdots+\beta_{n-1}+\gamma$. 
This process realizes $w$ as the free reduction of a product of
conjugates of parabolics. 
\end{proof}
\begin{corollary}\label{alternating_palindromes_fix_root}
 In $G_{C,R}$, the stabilizer of the root vertex $\bp$ is generated by palindromes
 of length $2R+1$, and these are the shortest nonempty words
 representing elements of $\stab(\bp)$. 
\end{corollary}
\begin{proof}
  The previous proposition shows that $\stab(\bp)$ is generated by
  elements that can be expressed as a word that travels along a
  geodesic from $\bp$ to some vertex on $\bdry T_{C,R}$, which is
  distance $R$ away, then uses one letter to reflect off a mirror, and
  then travels the original path back to $\bp$. Such a word has length
  $2R+1$ and reads the same forward as backward.
\end{proof}
This implies Biggs's girth bound:
\begin{corollary}
 $\girth(\cay(G_{C,R},C))\geq 2R+1$
\end{corollary}

\begin{proposition}\label{rainbow_cycle}
  If $C=\{c_1,\dots,c_{|C|}\}$ then 
  $w:=\prod_{i=1}^{|C|}c_i$ is an $N_{C,R}$--cycle.
\end{proposition}
\begin{proof}
  $T_{C,R}$ can be drawn in the plane in such a way that the (half)edges
      incident to each vertex are colored $c_1$, $c_2$, \dots, $c_{|C|}$,
      when read clockwise starting from the $c_1$ edge.
      Then the track of the $w^{N_{C,R}}$ action on $\bp$ can be imagined
      as following the boundary of a tubular neighborhood of the embedding of $T_{C,R}$ into the plane, starting
      and ending at $\bp$. See \fullref{fig:rainbow_cycle}.

      \begin{figure}[h]
        \centering
        \begin{tikzpicture}\tiny
          \coordinate (o) at (0,0);
          \coordinate (r) at (0:1);
          \coordinate (g) at (240:1);
          \coordinate (b) at (120:1);
          \coordinate (rb) at (-25:2);
          \coordinate (rg) at (25:2);
          \coordinate (gb) at (240+25:2);
          \coordinate (gr) at (240-25:2);
          \coordinate (bg) at (120-25:2);
          \coordinate (br) at (120+25:2);
          \coordinate (rgr) at (25-15:2.5);
          \coordinate (rgb) at (25+15:2.5);
          \coordinate (rbg) at (-25-15:2.5);
          \coordinate (rbr) at (-25+15:2.5);
          \coordinate (gbg) at (240+25-15:2.5);
          \coordinate (gbr) at (240+25+15:2.5);
          \coordinate (grb) at (240-25-15:2.5);
          \coordinate (grg) at (240-25+15:2.5);
          \coordinate (bgr) at (120-25-15:2.5);
          \coordinate (bgb) at (120-25+15:2.5);
          \coordinate (brb) at (120+25-15:2.5);
          \coordinate (brg) at (120+25+15:2.5);
          \draw[ultra thick, red] (o)--(r) (g)--(gr) (b)--(br);
          \draw[ultra thick, green] (o)--(g) (r)--(rg) (b)--(bg);
          \draw[ultra thick, blue] (o)--(b) (r)--(rb) (g)--(gb);
          \draw[ultra thick, red, -|] (rg)--(rgr);
          \draw[ultra thick, red, -|] (rb)--(rbr);
          \draw[ultra thick, green, -|] (rb)--(rbg);
          \draw[ultra thick, blue, -|] (rg)--(rgb);
          \draw[ultra thick, green, -|] (gb)--(gbg);
          \draw[ultra thick, red, -|] (gb)--(gbr);
          \draw[ultra thick, green, -|] (gr)--(grg);
          \draw[ultra thick, blue, -|] (gr)--(grb);
          \draw[ultra thick, blue, -|] (bg)--(bgb);
          \draw[ultra thick, red, -|] (bg)--(bgr);
          \draw[ultra thick, green, -|] (br)--(brg);
          \draw[ultra thick, blue, -|] (br)--(brb);
          \filldraw (o) circle (2pt) (r) circle (2pt) (g) circle (2pt)
          (b) circle (2pt) (rg) circle (2pt) (rb) circle (2pt) (gr)
          circle (2pt) (gb) circle (2pt) (br) circle (2pt) (bg) circle
          (2pt);
          \draw[red,>={Latex[length=5pt,width=5pt]},>-] ($(o)+(60:.2)$)--($(r)+(120:.2)$) node[at
          start, xshift=3pt, above] {0};
          \draw[green] ($(r)+(120:.2)$)--($(rg)+(150:.2)$);
          \draw[blue] ($(rg)+(150:.2)$)--(rgb)--($(rg)+(30:.2)$);
          \draw[red] ($(rg)+(30:.2)$)--(rgr) node[at start, right]
          {1} (rgr)--($(rg)+(-90:.2)$);
          \draw[green] ($(rg)+(-90:.2)$)--($(r)+(0:.2)$);
          \draw[blue] ($(r)+(0:.2)$)--($(rb)+(90:.2)$);
          \draw[red] ($(rb)+(90:.2)$)--(rbr) node[at start, above]
          {2} (rbr)--($(rb)+(-30:.2)$);
          \draw[green] ($(rb)+(-30:.2)$)--(rbg)--($(rb)+(210:.2)$);
          \draw[blue] ($(rb)+(210:.2)$)--($(r)+(240:.2)$);
          \draw[red] ($(r)+(240:.2)$)--($(o)+(-60:.2)$) node[at start, below]
          {3};
          \draw[green] ($(o)+(-60:.2)$)--($(g)+(0:.2)$);
          \draw[blue] ($(g)+(0:.2)$)--($(gb)+(30:.2)$);
          \draw[red] ($(gb)+(30:.2)$)--(gbr) node[at start, right]
          {4} (gbr)--($(gb)+(-90:.2)$);
          \draw[green] ($(gb)+(-90:.2)$)--(gbg)--($(gb)+(150:.2)$);
          \draw[blue] ($(gb)+(150:.2)$)--($(g)+(240:.2)$);
          \draw[red] ($(g)+(240:.2)$)--($(gr)+(330:.2)$) node[pos=0,xshift=-2pt, below]
          {5};
          \draw[green] ($(gr)+(330:.2)$)--(grg)--($(gr)+(200:.2)$);
          \draw[blue] ($(gr)+(200:.2)$)--(grb)--($(gr)+(90:.2)$);
          \draw[red] ($(gr)+(90:.2)$)--($(g)+(120:.2)$) node[at start, above]
          {6};
          \draw[green] ($(g)+(120:.2)$)--($(o)+(180:.2)$);
          \draw[blue] ($(o)+(180:.2)$)--($(b)+(240:.2)$);
          \draw[red] ($(b)+(240:.2)$)--($(br)+(-90:.2)$) node[at start, xshift=-2pt,below]
          {7};
          \draw[green] ($(br)+(-90:.2)$)--(brg)--($(br)+(150:.2)$);
          \draw[blue] ($(br)+(150:.2)$)--(brb)--($(br)+(30:.2)$);
          \draw[red] ($(br)+(30:.2)$)--($(b)+(120:.2)$) node[at start, above,xshift=3pt]
          {8};
          \draw[green] ($(b)+(120:.2)$)--($(bg)+(210:.2)$);
          \draw[blue] ($(bg)+(210:.2)$)--(bgb)--($(bg)+(90:.2)$);
          \draw[red] ($(bg)+(90:.2)$)--(bgr) node[at start, above]
          {9} (bgr)--($(bg)+(-30:.2)$);
          \draw[green] ($(bg)+(-30:.2)$)--($(b)+(0:.2)$);
          \draw[blue, >={Latex[length=5pt,width=5pt]},->] ($(b)+(0:.2)$)--($(o)+(60:.2)$);
        \end{tikzpicture}
        \caption{$T_{3,2}$ with the track of 
          $(rgb)^{10}$ at $\bp$, showing that $rgb$ is a 10--cycle.
        A vertex is numbered $n$ when that vertex is $\bp.(rgb)^n$. Mirrors are emphasized by terminal bars.}
        \label{fig:rainbow_cycle}
      \end{figure}
      
      The track hits each vertex $v$ of $T_{C,R}\setminus\{\bp\}$
      $C$--many times, once for each possible transition $c_i$ to
      $c_{i+1}$ or $c_{C}$ to $c_1$.
      However, only the $c_{C}$ to $c_1$ transition at $v$ corresponds to a power of $w$ taking
      $\bp$ to $v$, so $n\mapsto \bp.w^n$ gives a bijection between
      numbers $1,\dots, N_{C,R}$ and vertices of $T_{C,R}$. 
    \end{proof}
    This gives us a first girth estimate, as follows, which we will soon
    improve.
    \begin{corollary}
      $\girth(\cay(G_{C,R},C))\leq
    CN_{C,R}$
    \end{corollary}

For any subset $C'\subset C$ the tree $T_{C,R}$ decomposes into a
disjoint union of maximal $C'$--colored subtrees, where each component
is a copy of $T_{C',r}$ for some $r\leq R$.
The subgroup $G'$ of $G_{C,R}$ generated by $C'$ preserves each such
$C'$--colored component and maps onto the corresponding $G_{C',r}$,
simply by tracking the action of words on the component.
In general this is a proper quotient, since $G'$ is simultaneously
acting on several different $C'$--colored components, but this
observation allows us to describe the cycle structure of elements of
$G'$ in terms of induced actions on the $C'$--colored components.
When $C'=2$ we will also call a $C'$--colored component a
\emph{dichrome arc}. In this case it is an arc with both ends being
mirrors of $T_{C,R}$.

\begin{proposition}\label{dichrome_is_dihedral}
  $G_{2,R}$ is a dihedral group of order $2(2R+1)$ and
  $\stab(\bp)\cong\mathbb{Z}/2\mathbb{Z}$. 
  The alternating words $c_1c_2\dots$ and $c_2c_1\dots$ of length
  $2R+1$ represent the same element of $G_{2,R}$ and they act by
  fixing $\bp$ and exchanging the two halves of $T_{2,R}\setminus\{\bp\}$,  preserving
  distance to $\bp$.
\end{proposition}
\begin{proof}
  By \fullref{rainbow_cycle}, the element $c_1c_2\in G_{2,R}$ acts as a $(2R+1)$--cycle on the
  $2R+1$ vertices of the line graph $T_{2,R}$, and
  $c_1(c_1c_2)c_1=c_2c_1=(c_1c_2)^{-1}$. 
 The peripheral subgroups are isomorphic to $\mathbb{Z}/2\mathbb{Z}$,
 and $\stab(\bp)$ is generated by alternating words $w=c_1c_2\cdots$ and $w'=c_2c_1\cdots$ of length $2R+1$. 
 But $ww'=(c_1c_2)^{2R+1}=1$, and $w^2=(w')^2=1$, so
 $\stab(\bp)=\langle w\rangle\cong\mathbb{Z}/2\mathbb{Z}$. 
 Finally, $w$ is an involution in a centerless dihedral group, so
 $w(c_1c_2)w=(c_1c_2)^{-1}$, so $\bp.(c_1c_2)^nw=\bp.w(c_2c_1)^n=\bp.(c_2c_1)^n$.
 The distance claim follows from the symmetry exchanging the 
 colors. 
\end{proof}
\begin{corollary}\label{cor:dihedral_prime}
  If $2n+1$ is prime relative to $2R+1$ then the alternating words
  $x:=c_1c_2\cdots$ and $y:=c_2c_1\cdots$ of length $2n+1$ generate $G_{2,R}$.
\end{corollary}
\begin{proof}
  Since $c_1c_2$ has order $2R+1$ in $G_{2,R}$ and $2n+1$ is prime
  relative to $2R+1$, the element  $xy=(c_1c_2)^{2n+1}$ also has order
  $2R+1$.
  Since $G_{2,R}$ is dihedral of order $2(2R+1)$, it is generated by
  any involution and element of order $2R+1$.
  Thus, $\langle x,y\rangle=\langle x,xy\rangle=G_{2,R}$.
\end{proof}

\begin{proposition}\label{prop:rotation}
For  $w:=\prod_{i=1}^{C}c_i$, let $w^{N_{C,R}/C}$ denote the
initial subword of length $N_{C,R}$ of $w^{N_{C,R}}$.
Then $w^{N_{C,R}/C}$ acts as a rotation on $T_{C,R}$ that fixes
$\bp$ and takes edges of color $c_i$ to edges of color $c_{i+1}$. 
\end{proposition}
\begin{proof}
  For the planar embedding of $T_{C,R}$ as in
  \fullref{rainbow_cycle}, the word $w^{N_{C,R}/C}$ pushes every vertex of
  $T_{C,R}$ exactly $1/C$ of the way around the track of
  $w^{N_{C,R}}$. 
\end{proof}
\begin{corollary}\label{girth_upper_bound}
  $\girth(\mathrm{Cay}(G_{C,R},C))\leq 2N_{C,R}$
\end{corollary}
\begin{proof}
  By \fullref{prop:rotation}, $(\prod_{i=1}^{C}c_i)^{N_{C,R}/C}$
  is a rotation of $T_{C,R}$. By the same argument, 
  $(\prod_{i=1}^{C}c_{C+1-i})^{N_{C,R}/C}$
  is the opposite rotation.
  Since $N_{C,R}/C$ is not an integer, the former does not end with
  $c_{C}$, so their product is a freely reduced word of length
  $2N_{C,R}$ representing the
  identity permutation. 
\end{proof}

\begin{definition}
  Define some least common multiple functions:
  \[\lcmleq\from\mathbb{N}\to\mathbb{N}:n\mapsto\lcm\{m\mid 1\leq
    m\leq n\}\]
  \[\oddlcmleq\from\mathbb{N}\to\mathbb{N}:n\mapsto\lcm\{2m+1\mid
    1\leq 2m+1\leq n\}\]
\end{definition}

\begin{proposition}\label{lambda_girth}
  If $C=2$ or $R=1$ then 
  $\girth(\cay(G_{C,R},C))=4R+2$.
  Otherwise, $4R+3\leq \girth(\cay(G_{C,R},C))\leq 2\oddlcmleq(2R+1)$.
\end{proposition}
\begin{proof}
  The girth when $C=2$ follows from \fullref{cor:dihedral_prime},
  so assume $C>2$.
  Let $w$ be a word realizing the girth.
  If $w$ is dichrome then it acts separately on each of the dichrome
  arcs of the same two colors. These arcs have $2r+1$ vertices for each
  $0\leq r\leq R$, so \fullref{cor:dihedral_prime} 
gives $|w|=2\oddlcmleq(2R+1)\geq 2(2R+1)$, with equality only if
$R=1$. 

  Now suppose $w$ is not dichrome, and decompose it as $w=xyz$ where
  $x$ is its maximal dichrome prefix, which has length at least 2, and $y$
  is the maximal subword of length at most $R$ following $x$.
  By hypothesis, $y$ is nonempty, so there is some leaf $v$ such that the geodesic from $v$ to $\bp$ has
  color sequence that agrees with $y$ for the length of $y$.
  If $|y|< R$ then $z$ is empty.
  However, if $w=xy$
is a girth witness then $v.w=v$, but $v.x=v$, so $v.xy=v.y$ is the
vertex at distance $|y|>0$ from $v$ along the geodesic from $v$ to
$\bp$, which is a contradiction. Thus, $|y|=R$ and $z$ is nonempty.
The last letter of $y$ is different from the first letter of $z$, so
$z$ leaves $\bp$ on a different edge than the track of $xy$ from $v$
arrived on. Thus, for $v.xyz=v$, we must have that the track of $z$
from $\bp$ starts on a different branch than that containing $v$,
travels to a mirror, reflects, then does some more traveling before
eventually passing back through $\bp$ and continuing to $v$.
This gives $|z|\geq
3R+1$, so $|w|\geq 4R+3$.
\end{proof}
A similar lower bound argument appears in
\cite[Section~2.1]{otto2024acyclicityfinitegroupsgroupoids}.

We now compare the upper bounds given by \fullref{girth_upper_bound}
and \fullref{lambda_girth}.
The maximal power of 2 that occurs in numbers $x\leq n$ is $2^{\lfloor \log_2(n)\rfloor}$, so:
  \[\oddlcmleq(n)=\lcm\{2m+1\mid 1\leq 2m+1\leq
    n\}=\frac{\lcmleq(n)}{2^{\lfloor \log_2(n)\rfloor}}\]
A similar argument expresses the second Chebyshev function $\psi$ as
$\psi(n)=\ln(\lcmleq(n))$, or, equivalently, $\lcmleq(n)=e^{\psi(n)}$.
The Prime Number Theorem has an equivalent formulation as $\psi(n)\sim
n$, and this can be made effective, for example by the following
estimate, which will be used in \fullref{sec:main_theorem}:
\begin{theorem}[{\cite[Theorem~3.3]{Dus18}}]\label{dusart}
  For $x\geq 2$:
  \[|\psi(x)-x|<\frac{.85x}{\ln(x)}\]
\end{theorem}

\begin{lemma}\label{lem:upper_bound_comparison}
  Asymptotically in $R$ the bound $2N_{C,R}$ is smaller if $C\leq 8$ and
  $2\oddlcmleq(2R+1)$ is smaller if $C\geq 9$.
\end{lemma}
\begin{proof}
  \begin{align*}
    \ln\left(\frac{2\oddlcmleq(2R+1)}{2N_{C,R}}\right)
    &=\psi(2R+1)-\lfloor\log_2(2R+1)\rfloor\ln 2-\ln(N_{C,R})\\
    &=R(2-\ln(C-1))+o(R) 
  \end{align*}

Whether this goes to $\pm\infty$ as $R\to\infty$ hinges on
whether $C-1<e^2\approx 7.4$.
\end{proof}

\section{2--transitivity}\label{sec:primitivity}
\begin{proposition}\label{two_transitive}
  If $C>2$ then $G_{C,R}\act T_{C,R}$ is 2--transitive and not sharply 2--transitive.
\end{proposition}
\begin{proof}
  We already know $G_{C,R}\act T_{C,R}$ is transitive,
  so 2--transitivity is equivalent to transitivity of 
  $\mathrm{Stab}(\bp)\act T_{C,R}\setminus\{\bp\}$.

  Let $C=\{c_1,c_2,c_3,\dots\}$.
  Consider any two colors $c_i$ and $c_j$.
  The vertices $\bp.c_i$ and $\bp.c_j$ are contained in the
  $(c_i,c_j)$--dichrome arc through $\bp$.
  Words in $\langle c_i,c_j\rangle$ act on that dichrome arc as
  elements of $G_{2,R}$ act on $T_{2,R}$.
  In particular, by \fullref{dichrome_is_dihedral}, alternating words $c_ic_j\cdots$ and $c_jc_i\cdots$
  of length $2R+1$
  fix $\bp$ and exchange $\bp.c_i$ and $\bp.c_j$.
  Thus, $\stab(\bp)$ acts transitively on the set of vertices at
  height 1. 

Next we will show that for any vertex $v$ at height
$h>1$ there exists an element of $\stab(\bp)$ that moves $v$ to a
vertex at some lower positive height.
By induction, every vertex $v$ at positive height is in the same
$\stab(\bp)$ orbit as some vertex at height 1.
This will imply 2--transitivity, since we have already
shown that $\stab(\bp)$ is transitive on vertices at height 1. 

Let $v$ be a vertex at height $h>1$.
Suppose $c_1$ is the color of the down edge at $v_0:=v$ and $c_2$ is the
color of the down edge at $v_1:=v_0.c_1$.
Let $x:=c_1c_3\cdots$ and $y:=c_3c_1\cdots$ be alternating words of
length $2R+1$.
These are elements of $\stab(\bp)$, by
\fullref{alternating_palindromes_fix_root}. 
The $(c_1,c_3)$--dichrome arc $A_1$ through $v_0$ and $v_1$ has lowest vertex
$v_1$ and contains $2(R-(h-1))+1=2(R-h)+3$ vertices.
If $2(R-h)+3$ is prime relative to $2R+1$ then $x$ and $y$ generate the
full dihedral group $G_{2,R-h+1}$ acting on $A_1\cong T_{2,R-h+1}$, by \fullref{cor:dihedral_prime}.
In particular, there is an element of $\langle x,y\rangle$ that moves
$v_0$ to $v_1$, which is at height $h-1>0$, and we are done.

Suppose, to the contrary, that $D_1=\gcd(2(R-h)+3,2R+1)\geq 3$.
Consider the alternating word $z:=c_1c_3\cdots$ of length
$(2(R-h)+3)(2R+1)/D_1$.
Since $z$ is an alternating word whose length is an odd multiple of $2R+1$
we have $z\in\stab(\bp)$.
On the other hand, its length is an odd multiple of $2(R-h)+3$, so by 
\fullref{dichrome_is_dihedral} it fixes the lowest vertex on $A_1$ and
exchanges the two halves.
Thus, $v':=v_0.z=v_0.c_1c_3$ is in the same $\stab(\bp)$--orbit as
$v_0$.
If $|C|>3$ then the same construction applies using $c_i$ with
$3<i\leq |C|$ in place of $c_3$, so $v'$ is in the same
$\stab(\bp)$--orbit as $v_0.c_1c_i$. 
The up edges at $v_1$ are colored by $c_1$, $c_3$, and $c_i$ for all
$3<i\leq |C|$, so this shows that $\stab(\bp)$ acts transitively on the vertices
immediately above $v_1$.
Furthermore, odd numbers with a nontrivial common divisor differ by at
least $2D_1$, so $D_1=\gcd(2(R-h)+3,2R+1)\geq 3$ implies
$h\geq 4$.
Thus, the vertex $v_2:=v_0.c_1c_2$ is at height at least 2, which
implies that there is a down edge at $v_2$ leading to a vertex $v_3$.
That down edge  cannot be colored $c_2$, since $v_2$ has a
$c_2$--colored up edge, so without loss of generality it is
colored either $c_1$ or $c_3$.
In fact, we may assume it is colored $c_3$ by exchanging vertices $v_0$ and $v'$
and colors $c_1$ and $c_3$, if necessary. 

Let $A_2$ be the $(c_1,c_2)$--dichrome arc containing $v_0$, $v_1$,
and $v_2$.
By construction, the lowest vertex on $A_2$ is $v_2$, since
the down edge at $v_2$ is colored $c_3$.
The height of $v_2$ is $h-2$, so $A_2$ contains
$2(R-(h-2))+1=2(R-h)+5$ vertices.
Repeat the argument from $A_1$: if $2(R-h)+5$ is prime relative to
$2R+1$ then the alternating words $c_1c_2\cdots$ and $c_2c_1\cdots$ of
length $2R+1$ generate a subgroup of $\stab(\bp)$ that acts
transitively on $A_2$, so $v_0$ is in the same $\stab(\bp)$--orbit
as the vertex $v_2$, which is at lower positive height. 

The alternative is that $D_2=\gcd(2R+1,2(R-h)+5)\geq 3$.
As in the previous case, there are elements of $\stab(\bp)$ that fix
$v_2$ and exchange the two halves of $A_2$.
The same holds for every dichrome arc whose lowest vertex is $v_2$, so
$\stab(\bp)$ acts transitively on the upper neighbors of $v_2$.
Similarly, each such vertex is the lowest vertex on some dichrome arcs
containing $2(R-h)+3$ vertices, and there are elements of $\stab(\bp)$
reflecting such an arc across its lowest vertex.
Conclude that $\stab(\bp)$ acts transitively on the set of vertices $v$
in $T_{C,R}$ such that $v_2$ is the vertex at distance 2 from $v$ on
the unique geodesic from $v$ to $\bp$. 
Thus, without changing $\stab(\bp)$--orbits, we may choose the two
vertices above $v_2$ more carefully, as follows. 
Let $c_i$ be the color of the down edge at $v_3$.
We know that $c_i\neq c_3$, since $v_3.c_3=v_2$.
Choose $j$ different from $i$ and $3$, which is possible since $C\geq 3$.

By the transitivity just established, we may replace $v_0$ by $v_2.c_jc_3$ without
changing the $\stab(\bp)$--orbit of $v_0$ or the fact that $v_2$ lies
on the geodesic from $v_0$ to $\bp$.
After this replacement, the geodesic from $v_0$ to $v_3$ has color
sequence $c_3c_jc_3$.
Moreover, the down edge at $v_3$ has color $c_i$, which is different
from both $c_3$ and $c_j$, so the $(c_j,c_3)$--dichrome arc through
$v_0$ has lowest vertex $v_3$. 

In summary, we iteratively find dichrome arcs $A_1,\dots$ such
that $A_n$ contains 
$2(R-h)+2n+1$ vertices, one of which is in the same
$\stab(\bp)$--orbit as $v$, and whose lowest vertex $v_n$ is at height
$h-n$.
We also have that for all $m<n$:
\[D_m:=\gcd(2(R-h)+2m+1,2R+1)>1\]
The iteration stops if $D_n=1$, which happens eventually
since $D_{h-1}=1$.
When $D_n=1$ we use alternating words in the two colors of $A_n$ to
construct an element of $\stab(\bp)$ that moves $v$ to $v_n$, which is strictly
closer to $\bp$. 
On the other hand, if $D_m>1$ for all $m\leq n$ then the iterated
transitivity argument implies that all vertices $v'$ for
which $v_n$ is the vertex at distance $n$ from $v'$ on the geodesic
from $v'$ to $\bp$ are in the same $\stab(\bp)$--orbit. 
In particular, since $C\geq 3$ we can look at the first two colors
$c_k$ and $c_\ell$ on the geodesic from $v_n$ to $\bp$ and choose a
third color $c_{\ell'}$, distinct from $c_k$ and $c_{\ell}$.
Then the $(c_k,c_{\ell'})$--dichrome arc $A_{n+1}$ through $v_n$ has lowest
vertex $v_{n+1}:=v_n.c_k$ exactly one step below $v_n$.
Replace $v$ by the unique vertex $v'$ that is on $A_{n+1}$ and such
that $v_n$ is distance $n$ from $v'$ along the geodesic from $v'$ to
$\bp$. Since $v'$ is in the same
$\stab(\bp)$ orbit as $v$, and $(n+1)$--steps above $v_{n+1}$, this sets up the $n+1$ step of the iteration.

\fullref{fig:transitive} illustrates an example in which several
lateral exchanges are necessary before the vertex can be moved down. 
\begin{figure}[h]
        \centering
        \begin{tikzpicture}
          \coordinate (o) at (0,0);
          \coordinate (r) at (-90:1);
          \coordinate (g) at (135:1);
          \coordinate (b) at (45:1);
          \coordinate[label={[label distance=0pt] 180:$v''$}] (gb) at ($(g)+(90:1)$);
          \coordinate[label={[label distance=0pt] -90:$v'''$}] (gr) at ($(g)+(180:1)$);
          \coordinate[label={[label distance=0pt] 0:$v'$}]  (bg) at ($(b)+(90:1)$);
          \coordinate[label={[label distance=0pt] 0:$v$}] (br) at
          ($(b)+(0:1)$);
          \coordinate (rg) at ($(r)+(150:1)$);
          \coordinate (gbg) at ($(gb)+(60:.5)$);
          \coordinate (gbr) at ($(gb)+(120:.5)$);
          \coordinate (grb) at ($(gr)+(150:.5)$);
          \coordinate (grg) at ($(gr)+(210:.5)$);
          \coordinate (bgr) at ($(bg)+(60:.5)$);
          \coordinate (bgb) at ($(bg)+(120:.5)$);
          \coordinate (brb) at ($(br)+(-60:.5)$);
          \coordinate (brg) at ($(br)+(60:.5)$);
          \draw[ultra thick, red] (o)--(r) (g)--(gr);
          \draw[ultra thick, red] (b)--(br);
          \draw[ultra thick,blue] (r)--($(r)-(0,1)$);
          \draw[ultra thick, green] (o)--(g) (b)--(bg);
          \draw[ultra thick,green] (r)--(rg);
          \draw[ultra thick, blue] (o)--(b) (g)--(gb);
          \draw[ultra thick, green, -|] (gb)--(gbg);
          \draw[ultra thick, red, -|] (gb)--(gbr);
          \draw[ultra thick, green, -|] (gr)--(grg);
          \draw[ultra thick, blue, -|] (gr)--(grb);
          \draw[ultra thick, blue, -|] (bg)--(bgb);
          \draw[ultra thick, red, -|] (bg)--(bgr);
          \draw[ultra thick, green, -|] (br)--(brg);
          \draw[ultra thick, blue, -|] (br)--(brb);
          \filldraw (o) circle (2pt) (r) circle (2pt) (g) circle (2pt)
          (b) circle (2pt) (gr)
          circle (2pt) (gb) circle (2pt) (br) circle (2pt) (bg) circle
          (2pt);
\begin{scope}[dashed,decoration={
    markings,
    mark=at position 0.5 with {\arrow{>}}}
    ] 
    \draw[postaction={decorate}] (br)--(bg);
    \draw[postaction={decorate}] (bg)--(gb);
    \draw[postaction={decorate}] (gb)--(gr);
    
\end{scope}        
        \end{tikzpicture}
        \caption{In $G_{3,7}$ suppose $v$ is at height
          7 on the $(r,b)$--dichrome arc containing $\bp$.
          All alternating dichrome words of length 15 fix $\bp$, but none
          of them move $v$, $v'$, or $v''$ closer to $\bp$, since the
          dichrome arcs on which these vertices sit have lengths dividing 15. 
          However, an $(r,g)$ alternating word of length 15 swaps $v$ and
          $v'$, a $(b,g)$ alternating word of length 15 swaps $v'$ and
          $v''$, and an $(r,b)$ alternating word of length 15 swaps $v''$
          and $v'''$. Finally, the $(r,g)$--dichrome arc containing
          $v'''$ has 7 vertices, prime relative to 15. The alternating
          word $rg\cdots r$ of length 15 moves $v'''$ closer to
          $\bp$.
}
        \label{fig:transitive}
      \end{figure}

        Finally, $G_{C,R}$ is not sharply 2--transitive, since the alternating
word $c_1c_2\cdots$ of length $2R+1$ fixes $\bp$ and every vertex of
$\bdry_{c_3}T_{C,R}$, but is nontrivial, since it exchanges the two
sides of the $(c_1,c_2)$--dichrome arc through $\bp$.
\end{proof}

\section{Proof of the main theorem}\label{sec:main_theorem}
  The tree $T_{C,R}$ has $N_{C,R}$--many vertices and  $(N_{C,R}-1)/C=E_{C,R}$--many edges
of each color, for:
   \[N_{C,R}=
  \begin{cases}
    2R+1&\text{if }C=2\\
    \frac{C(C-1)^R-2}{C-2}&\text{if }C>2
  \end{cases}
  \quad\quad
  E_{C,R}=
  \begin{cases}
    R&\text{if }C=2\\
    \frac{(C-1)^R-1}{C-2}&\text{if }C>2
  \end{cases}
\]

Since each generator of $G_{C,R}$ is a product of $E_{C,R}$--many
disjoint transpositions,  $G_{C,R}<\alt(N_{C,R})$ if and only
if $E_{C,R}$ is even.
But $E_{C,R}$ is odd if $C$ is odd, and if $C$ is even then the parity of $E_{C,R}$ alternates.

For $R=1$, the generators plus the elements used to permute neighbors of
the root in the proof of \fullref{two_transitive} give all the
transpositions, so $G_{C,1}\cong\sym(C+1)$, so assume $R>1$.

By \fullref{rainbow_cycle}, the word $w=\prod_{c\in C}c$ is an $N_{C,R}$--cycle,
fixing 0 points.
Apply \fullref{jones}. 
We are not in \fullref{jones} case~\eqref{item:sporadic}, since
11 and 23 do not occur as possible values of $N_{C,R}$ when $R>1$.
We are not in \fullref{jones} case~\eqref{item:affine}, because $\mathrm{AGL}_1(p)$ is sharply 2--transitive, while $G_{C,R}$ is not.
This leaves the possibility of \fullref{jones}
case~\eqref{item:projective}, that $G_{C,R}$ is a subgroup of some
projective semi-linear group with the same permutation degree. 

One way to rule out the projective case is to produce cycles in
$G_{C,R}$ with many fixed points, which we can do in certain cases.

Suppose that $C'\subset C$ is a proper subset containing at least two
colors.
Consider the word $w=\prod_{c\in C'}c$.
By the argument of \fullref{rainbow_cycle}, $w$ acts cyclically on
$C'$--colored components of $T_{C,R}$.
These are disjoint subtrees whose sizes are $N_{C',r}$ for each $0\leq
r\leq R$, so the cycle decomposition of $w$ consists of cycles of all of
these sizes. Let $m:=\prod_{1\leq r<R}N_{C',r}$.
Then $w^m$ acts trivially on all $C'$--colored subtrees except
possibly the
unique largest one, the one
containing $\bp$, of size $N_{C',R}$.
If $N_{C',R}$ is prime relative to $N_{C',r}$ for all $r<R$ then 
$N_{C',R}$ is prime relative
to $m$, so $w^m$ is an $N_{C',R}$--cycle.
Thus, the cycle decomposition of $w^m$ is a single $N_{C',R}$--cycle,
and, 
since $R>1$, $N_{C,R}-N_{C',R}>2$, so the cycle $w^m$ fixes more than 2 vertices.
Then \fullref{jones} says $G_{C,R}$ contains $\alt(N_{C,R})$.
This gives:
\begin{equation}
  \label{eq:primary_cycle}
  \parbox{0.8\linewidth}{
   If there is $C'\subset C$ with $2\leq C'<C$ and $N_{C',R}$ is
   prime relative to $N_{C',r}$ for all $r<R$ then $G_{C,R}$ contains $\alt(N_{C,R})$.
  }
\end{equation}
This condition is true, in particular, if $N_{2,R}=2R+1$ is prime, so
for all $C>2$ there are infinitely many $R$, including all $R\leq 3$
for which $G_{C,R}$ contains $\alt(N_{C,R})$.

Now consider $C'=C-1$, and consider the $C'$--colored subtrees of
$T_{C,R}$.
As before, there is a unique largest one, containing $\bp$, but this
time there is also a unique second largest one, rooted at $\bp.c$ for
the unique $c\in C\setminus C'$.
Thus, for $w$ the product of colors of $C'$, the cycle decomposition
of $w$ consists of cycles of sizes $N_{C',r}$ for all $0\leq r\leq R$,
and among these there is a unique $N_{C',R-1}$ cycle.
Let $m:=\prod_{1\leq r\leq R,\,r\neq R-1}N_{C',r}$.
If $N_{C-1,R-1}$ is prime relative to all of the other $N_{C-1,r}$
then  $w^m$ is an $N_{C-1,R-1}$ cycle, and
$N_{C,R}-N_{C-1,R-1}>2$, so \fullref{jones} says $G_{C,R}$ contains $\alt(N_{C,R})$.
\begin{equation}
  \label{eq:secondary_cycle}
  \parbox{0.8\linewidth}{
    If $N_{C-1,R-1}$ is prime relative to $N_{C-1,r}$ for all
    $r<R-1$ and for $r=R$, then $G_{C,R}$ contains $\alt(N_{C,R})$.
  }
\end{equation}

\fullref{tab:N} gives the prime factorizations of $N_{C,R}$.
By \eqref{eq:primary_cycle}, if an entry $N_{C,R}$ is prime
relative to every entry above it in its column then $G_{D,R}$ contains
$\alt(N_{D,R})$ for all $D>C$, and the entries $N_{D,R}$ to the right of $N_{C,R}$ in
its row are colored violet in
\fullref{tab:N}.
For example all of the row $R=2$ for $C\geq 3$ is violet because $N_{2,2}=5$ is
prime relative to $N_{2,1}=3$. 

Condition \eqref{eq:secondary_cycle} applies to say that
$G_{C,R}$ contains $\alt(N_{C,R})$ whenever we look at the entry $N_{C-1,R-1}$ above and to
the left of $N_{C,R}$, and if that entry is prime relative to the
entry $N_{C-1,R}$ immediately below it and to every entry $N_{C-1,r}$
above it in its column.
We color such entries $N_{C,R}$ in \fullref{tab:N} blue if they are
not already violet.
For example  $G_{3,4}$ contains $\alt(N_{3,4})$, and entry $R=4$, $C=3$ is blue, because 7
is prime relative to 9 and to both of 3 and 5.

  \begin{table}[h]
    \centering
    {\tiny
    \begin{tabular}{|c||c|c|c|c|c|c|}
        \hline			
  $_R\backslash^C$&2&3&4&5&6&7\\
      \hline
      1&3&\color{violet}$2^2$&\color{violet}5&\color{violet}$2\cdot 3$&\color{violet}7&\color{violet}$2^3$\\
      2&5&\color{violet}$2\cdot 5$&\color{violet}17&\color{violet}$2\cdot
                                                   13$&\color{violet}37&\color{violet}$2\cdot
      5^2$\\
      3&7&\color{violet}$2\cdot 11$&\color{violet} 53&\color{violet}$2\cdot 53$&\color{violet}$11\cdot 17$&\color{violet}$2\cdot
      151$\\
      4&$3^2$&\color{blue}$2\cdot 23$&\color{red}$7\cdot 23$&\color{violet}$2\cdot
                                                   3\cdot
                                                   71$&\color{violet}
                                                        937&\color{violet}$2\cdot
      907$\\
      5&11&\color{violet}$2\cdot 47$&\color{violet}$5\cdot
                                     97$&\color{violet}$2\cdot
                                          853$&\color{violet}$43\cdot
                                                109$&\color{violet}$2\cdot
      5443$\\
      6&13&\color{violet}$2\cdot 5\cdot 19$&\color{violet}$31\cdot 47$&\color{violet}$2\cdot 3413$&\color{violet}$23\cdot
                                                          1019$&\color{violet}$2\cdot
      11\cdot 2969$\\
      7&$3\cdot 5$&\color{blue}$2\cdot 191$&\color{red}4373&\color{violet}$2\cdot 3^2\cdot 37\cdot
                                      41$&\color{violet}$7\cdot
                                           16741$&\color{violet}$2\cdot
      5\cdot 39191$\\
      8&17&\color{violet}$2\cdot 383$&\color{violet} 13121&\color{violet}$2\cdot 13\cdot 4201$&\color{violet}$11\cdot
                                                     53267$&\color{violet}$2\cdot
      673\cdot 1747$\\
      9&19&\color{violet}$2\cdot 13\cdot 59$&\color{violet}$5\cdot 7873$&\color{violet}$2\cdot
                                              218453$&\color{violet}
                                                       2929687&\color{violet}$2\cdot
      7054387$\\
      10&$3\cdot 7$&\color{blue}$2\cdot 5\cdot 307$&\color{red}$7\cdot 16871$&\color{red}$2\cdot 3\cdot
                                                        291271$&\color{red}$1487\cdot
                                                                 9851$&\color{violet}$2\cdot
      13\cdot 3255871$\\
      11&23&\color{violet}$2\cdot 37\cdot 83$&\color{violet}$19\cdot 29\cdot 643$&\color{violet}$2\cdot
                                                       3495253$&\color{violet}$59\cdot
                                                                 349\cdot
                                                                 3557$&\color{violet}$2\cdot
      67\cdot 631\cdot 6007$\\
      12&$5^2$&\color{blue}$2\cdot 6143$&\color{red}1062881&\color{violet}$2\cdot 13981013$&\color{violet}$577\cdot
                                                        634681$&\color{violet}$2\cdot
      5\cdot 163\cdot 271\cdot 6899$\\
  \hline  
\end{tabular}}
    \caption{Prime factorizations of $N_{C,R}$.}
    \label{tab:N}
  \end{table}
  For the remaining 6 entries in the table with $C>2$,
which are colored red, we checked by brute force computation that
there is no $d\geq 2$ and prime power $q$ such that
$(q^d-1)/(q-1)=N_{C,R}$.
However, we cannot always exclude such an equality. 
For example, when $C$ is a prime power, $q=C^2$, and $R=d=2$ there is a
coincidental equality $N_{C,R}=(q^d-1)/(q-1)$.
These particular examples do not bother us, since we have already
established the $R=2$ case of the theorem, but we do not know
if there are any other such coincidences for larger $R$, so we turn to
another approach for ruling out $G_{C,R}\leq \pgammal_d(q)$.
The idea is to demonstrate an element of $G_{C,R}$ whose order is too
high relative to $N_{C,R}$ for it to occur in a projective semi-linear
group of the same permutation degree.

For an element $g$ in a group $G$, let $\ord(g)$ be the order of $g$
in $G$, and let $\meo(G):=\max\{\ord(g)\mid g\in G\}$.
The group $\pgammal_d(q)$ is a semi-direct product of $\pgl_d(q)$ with
the Galois group of the field extension of $\mathbb{F}_q$ over
$\mathbb{F}_p$.
It embeds into $\aut(\psl_d(q))$.

If $G_{C,R}\leq \pgammal_d(q)$ then
$N_{C,R}=\frac{q^d-1}{q-1}$.
The exceptional cases $(d,q)=
(2,4),(3,2)$ cannot occur in this setting, since in those cases we do not have
$N_{C,R}=\frac{q^d-1}{q-1}$  for $C>2$ and $R>2$.
Outside the exceptional cases, the values of $\meo(\pgl_d(q))$ and
$\meo(\aut(\psl_d(q)))$ agree, and are equal to $\frac{q^d-1}{q-1}$ \cite[Corollary~2.7 and Theorem~2.16/Table~3]{MR3391897}, so:
\begin{equation}
  \meo(\pgammal_d(q))=\frac{q^d-1}{q-1}=N_{C,R}
\label{eq:meo}
\end{equation}
This gives a bound for  $\ln(\meo(\pgammal_d(q)))$, when
$\pgammal_d(q)$ is of degree $N_{C,R}$, that
asymptotically grows linearly in $R$ with slope $\ln(C-1)$.

The goal for the remainder of this section is to identify an element of
$G_{C,R}$ whose order is greater than the maximum given in
\eqref{eq:meo}, which is impossible if $G_{C,R}\leq \pgammal_d(q)$, so
we have $G_{C,R}\not\leq \pgammal_d(q)$.

First specialize to $C=3$.
Consider the element $c_1c_2\in G_{3,R}$.
As in the dichrome case of the proof of \fullref{lambda_girth}, the
element \(c_1c_2\) has order \(\oddlcmleq(2R+1)\). Thus
\fullref{dusart} gives:
\begin{align*}
\ln(\ord(c_1c_2))&=\ln\left(\frac{\exp(\psi(2R+1))}{2^{\lfloor
              \log_2(2R+1)\rfloor}}\right)\\
  &>\ln\left(\frac{\exp((1-\frac{.85}{\ln(2R+1)})(2R+1))}{2R+1}\right)
\end{align*}

Define:
\[F(R):=\ln\left(\frac{\exp((1-\frac{.85}{\ln(2R+1)})(2R+1))}{2R+1}\right)-\ln(N_{3,R})\]
Then $G_{3,R}\not\leq \pgammal_d(q)$ if $F(R)>0$. 
The $\ln(\ord(c_1c_2))$ term grows linearly in $R$ with slope 2, which is
higher than the asymptotic slope of the $\ln(\meo(\pgammal_d(q)))$
term, so $F(R)>0$ for sufficiently large $R$.
It is a calculus exercise to check that $F(R)$ is increasing for
$R\geq 1$. 
\fullref{fig:landau_plot_3} shows that $F$ already becomes positive at
$R= 5$, so we conclude that it stays positive for all $R\geq 5$.
Thus, $G_{3,R}$ is not a subgroup of $\pgammal_d(q)$.

\begin{figure}[ht]
\centering
\begin{tikzpicture}\tiny
  \begin{axis}[
    width=0.8\textwidth,
    height=4cm, 
    xlabel={$R$},
    ylabel={},
    xmin=1, xmax=9,
    xtick={1,2,3,4,5,6,7,8,9},
    legend pos=north west,
]
\addplot[mark=*,
mark size=1pt,
color=blue,
thick
]
coordinates {
(1, -0.4197223165665448)
(2, 0.7498886156875494)
(3, 1.99639471384467)
(4, 3.3211103808161275)
(5, 4.702851867384926)
(6, 6.126973382511817)
(7, 7.583765292269888)
(8, 9.066570666097986)
(9, 10.570645179727554)
};
\addlegendentry{lower bound for $\ln(\ord(c_1c_2))$}

\addplot[mark=x,
mark size=1.5pt,
color=red,
thick
] coordinates {
  (1, 1.3862943611198906)
(2, 2.302585092994046)
(3, 3.091042453358316)
(4, 3.828641396489095)
(5, 4.543294782270004)
(6, 5.247024072160486)
(7, 5.945420608606575)
(8, 6.641182169740591)
(9, 7.335633981927201)
};
\addlegendentry{$\ln(N_{3,R})$}

\end{axis}
\end{tikzpicture}
\caption{The lower bound for $\ln(\meo(G_{3,R}))$ is already greater
  than the upper bound for $\ln(\meo(\pgammal_d(q)))$ when $R=5$.}
\label{fig:landau_plot_3}
\end{figure}

\medskip

We do not push this computation to higher numbers of colors.
Indeed, for $C\geq 9$ the asymptotic slope $\ln(C-1)$ of
$\ln(\meo(\pgammal_d(q)))$ is greater than the asymptotic slope
$2$ of $\ln(\ord(c_1c_2))$. 
However, now that we know that $G_{3,R}$ is the symmetric
group, it turns out that the choice of the explicit element $c_1c_2$
was not a particularly good one in terms of establishing a lower bound
on $\meo(G_{C,R})$. 
The maximum element order of $\sym(n)$ is equal to the maximum,
over partitions of $n$, of the least common multiple of the numbers in
the partition.
The function\footnote{The Landau function is usually denoted $g(n)$ in
the literature.} $\Landau\from n\mapsto \meo(\sym(n))$ is called the
\emph{Landau function}, after Landau, who proved
\cite{landau1909handbuch} that $\ln(\Landau(n))\sim\sqrt{n\ln(n)}$.
 There are effective lower bounds for sufficiently large inputs, cf
 \cite[(1.6)]{MR979940}:
 \begin{equation}
   \label{eq:Landau}
   \ln(\Landau(n))\geq\sqrt{n\ln(n)}\quad\quad\text{ for }n\geq 906
 \end{equation}
There are also several thousand precomputed values for small $n$ \cite[\href{https://oeis.org/A000793}{A000793}]{oeis}.

For $C>3$, we induct on the number of colors: Suppose that for
$2<C'<C$ we know that $G_{C',R}$ contains $\alt(N_{C',R})$ for all
sufficiently large $R$.
Let $H$ be the subgroup of $G_{C,R}$ generated by a set $C'$
consisting of $C-1$
colors.
Then $H$ surjects onto $G_{C',R}$ and $G_{C',R-1}$ via the
action on the largest and second largest $C'$--colored components of $T_{C,R}$.
If $C'$ or $R$ is odd let $R':=R$, and if both are even let
$R':=R-1$, so that by the induction hypothesis, once $R$ is large
enough, $G_{C',R'}\cong\sym(N_{C',R'})$.
Take an element $h$ of maximal order $\Landau(N_{C',R'})$ in
$G_{C',R'}$, and lift it to an element $\tilde h\in H$.
Then $\ord(\tilde h)$ is a multiple of $\ord(h)$, so:
\[\ln(\meo(G_{C,R}))\geq\ln(\ord(\tilde
  h))\geq\ln(\Landau(N_{C',R'}))\]
  Thus, when $R$ is large enough that \eqref{eq:Landau} applies,
  $\ln(\meo(G_{C,R}))$ is bounded below by an exponential function of
  $R$:
  \[\ln(\meo(G_{C,R}))\geq\sqrt{N_{C',R'}\ln(N_{C',R'})}\geq
    (\sqrt{C-2})^{R-1}\]
  On the other hand, 
  $\ln(\meo(\pgammal_d(q)))\sim R\ln(C-1)$ is linear.
  Conclude that for sufficiently large $R$, the group $G_{C,R}$ contains
  elements whose orders are too high in relation to $N_{C,R}$ to belong
  to a projective semi-linear group of the same permutation degree.

To close the remaining gap and show that the desired conclusion holds
for all $R$, it is convenient to handle the case of large $C$ first. 
For $C> 10$ and $R\geq 3$ we have $N_{C-1,R'}>906$, so the bound of
\eqref{eq:Landau} gives:
\[Q(C,R):=\frac{\sqrt{N_{C-1,R'}\ln(N_{C-1,R'})}}{\ln(N_{C,R})}\leq
  \frac{\ln(\Landau(N_{C',R'}))}{\ln(N_{C,R})}\leq
  \frac{\ln(\meo(G_{C,R}))}{\ln(\meo(\pgammal_d(q)))}\]
Then $N_{C-1,R'}>(C-2)^{R-1}$ and $\ln(N_{C,R})<(R+1)\ln(C-1)$, so:
\[Q(C,R)>\frac{(C-2)^{(R-1)/2}\sqrt{(R-1)\ln(C-2)}}{(R+1)\ln(C-1)}\]
The right-hand side is increasing in $R$, increasing in $C$ for
$C\geq 11$, and greater than 1 at $R=3$ and $C=11$.
Thus, $Q(C,R)>1$ for all $C\geq 11$ and $R\geq 3$, and we conclude
that $G_{C,R}\not\leq\pgammal_d(q)$ for these ranges of $C$ and $R$. 

For $3<C\leq 10$, the above argument applies once $R$ is large
enough that $N_{C-1,R'}\geq 906$, so that the estimate
\eqref{eq:Landau} applies.
In each of these cases this happens before $R=13$, so since we already
know $G_{C,R}\not\leq\pgammal_d(q)$ for $R<13$, this completes the
proof. 

  Figures~\ref{fig:landau_plot_4} and \ref{fig:landau_plot_5} use
  precomputed values of $\Landau$ to illustrate that
  $\ln(\Landau(N_{C',R'}))$ is already much larger than $\ln(N_{C,R})$ for small values of $R$ when $C=4,5$.
  \begin{figure}[ht]
\centering
\begin{tikzpicture}\tiny
  \begin{axis}[
    width=0.8\textwidth,
    height=3.5cm, 
    xlabel={$R$},
    ylabel={},
    xmin=1, xmax=7,
    xtick={1,2,3,4,5,6,7},
    legend pos=north west,
]
\addplot[mark=*,
mark size=1pt,
color=blue,
thick
] coordinates {
(1, 1.3862943611198906)
(2, 3.4011973816621555)
(3, 6.040254711277414)
(4, 11.003099341537322)
(5, 18.76356637075074)
(6, 29.607900673056914)
(7, 46.153176282243166)
};
\addlegendentry{$\ln(\Landau(N_{3,R}))$}

\addplot[mark=x,
mark size=1.5pt,
color=red,
thick
] coordinates {
  (1, 1.6094379124341003)
(2, 2.833213344056216)
(3, 3.970291913552122)
(4, 5.081404364984463)
(5, 6.184148890937483)
(6, 7.284134806195205)
(7, 8.38320455141292)
};
\addlegendentry{$\ln(N_{4,R})$}

\end{axis}
\end{tikzpicture}
\caption{The lower bound for $\ln(\meo(G_{4,R}))$ is already greater
  than the upper bound for $\ln(\meo(\pgammal_d(q)))$ when $R\geq 2$.}
\label{fig:landau_plot_4}
\end{figure}
 \begin{figure}[ht]
\centering
\begin{tikzpicture}\tiny
  \begin{axis}[
    width=0.8\textwidth,
    height=3.5cm, 
    xlabel={$R$},
    ylabel={},
    xmin=1, xmax=7,
    xtick={1,2,3,4,5,6,7},
    legend pos=north west,
]
\addplot
[mark=*,
mark size=1pt,
color=blue,
thick
]coordinates {
(1, 1.791759469228055)
(2, 1.791759469228055)
(3, 12.794858810765376)
(4, 12.794858810765376)
(5, 53.82887695186044)
(6, 53.82887695186044)
(7, 197.242028734783)
};
\addlegendentry{$\ln(\Landau(N_{4,R'}))$}

\addplot
[mark=x,
mark size=1.5pt,
color=red,
thick
]coordinates {
(1, 1.791759469228055)
(2, 3.258096538021482)
(3, 4.663439094112067)
(4, 6.054439346269371)
(5, 7.441906728051625)
(6, 8.828494129466652)
(7, 10.214861737244696)
};
\addlegendentry{$\ln(N_{5,R})$}

\end{axis}
\end{tikzpicture}
\caption{The lower bound for $\ln(\meo(G_{5,R}))$ is already greater
  than the upper bound for $\ln(\meo(\pgammal_d(q)))$ when $R\geq 3$.}
\label{fig:landau_plot_5}
\end{figure}

\section{Further remarks}\label{sec:further_remarks}
We conclude with a few remarks on some well-known open questions
surrounding Biggs tree groups for $C>2$.

\begin{question}
  What is the precise girth of $\cay(G_{C,R},C)$?
\end{question}
There is quite a wide gap between the linear, in terms of $R$, girth lower bound and near exponential upper bounds given by \fullref{girth_upper_bound} and \fullref{lambda_girth}.

\begin{question}
  What is the diameter of $\cay(G_{C,R},C)$?
\end{question}
Using the main theorem:
$\cay(G_{C,R},C)$ has at least $N_{C,R}!/2$ vertices.
If $D$ is the diameter, the ball of radius $D$ contains all of these. 
A ball of radius $r$ in a $C$--regular graph contains at most the number
$N_{C,r}$ of vertices in a ball in the tree of the same valence.
For fixed $C$ and growing $R$, this gives a lower \emph{Moore bound} on $D$ by a constant multiple of
$R(C-1)^R$. 

\begin{proposition}\label{not_dg_bounded}
  For any $C>2$ and sequence of radii $R\to\infty$ the Cayley graphs
  $\cay(G_{C,R},C)$ do not form a bounded diameter-to-girth ratio family. 
\end{proposition}
\begin{proof}
 The family $\cay(G_{C,R},C)$ consists of $C$--regular graphs with girth
going to infinity, but since the order of the groups grows so fast, even with the most optimistic bounds on
diameter and girth the ratio cannot stay bounded:
\[\frac{\diam(\cay(G_{C,R},C))}{\girth(\cay(G_{C,R},C))}\cgeq\frac{R(C-1)^R}{(C-1)^R}\xrightarrow{R\to\infty}
  \infty\]
Here the estimates on the diameter and girth are from the Moore bound
and \fullref{girth_upper_bound}, respectively. 
\end{proof}

\begin{question}
  Is there a $C>2$ such that for some sequence of $R$'s the family
  $\cay(G_{C,R},C)$ is an expander family?
\end{question}
The exponential diameter estimate above is in terms of $R$, not
$|G_{C,R}|$, so it is not necessarily an obstacle to being an
expander family.

\medskip

Finally, one might wonder about other Biggs groups; that is, what is
the permutation group defined by some other colored graph, instead of a finite symmetric tree?
In that case, we could again add mirrored half-edges at every vertex
as necessary so that every vertex has exactly one incident (half)edge of each
color.
Then the procedure of tracking the action of words works just as in the tree case,
and this point of view may be useful in the analysis of the resulting
group.
However, both the proof of 2--transitivity and the production
of cycles depended heavily on the symmetry of $T_{C,R}$.
Moreover, every finite group generated by involutions occurs as a Biggs
group: take the colored Schreier graph of its right regular action
with respect to those generating involutions---this gives the original
group as the story of the picture.
Thus, rather than seeking uniform results for arbitrary pictures, a
more natural approach may be to identify special families of colored
graphs whose structure makes computation of the story tractable.

\section*{Acknowledgements}
This research was supported in part by the Austrian Science
Fund (FWF) \href{https://doi.org/10.55776/PAT7799924}{10.55776/PAT7799924} and
the Austria-Slovakia research cooperation grant ``Constructions of
expanders and extremal graphs'', OeAD WTZ SK 14/2024 and APVV
SK-AT-23-0019.

I first learned the question of determining the isomorphism types of the Biggs tree
groups from Goulnara Arzhantseva in 2014.
At that point I worked out the conditions \eqref{eq:primary_cycle} and
\eqref{eq:secondary_cycle}  of the main
theorem, but was dissatisfied with the incompleteness
of the 
classification and with \fullref{not_dg_bounded},  so the
project was shelved until Robert Jajcay proposed the same problem in a 2024 Workshop
of the bilateral cooperation project mentioned above.
I thank the two of them, as well as Tatiana Jajcayov\'{a}, for
organizing the cooperation project and encouraging me to write this
paper, and for comments on an earlier draft.


\bibliographystyle{hypersshort}
\bibliography{Biggs}

\end{document}